\documentclass[review]{elsarticle}

\usepackage{lineno,hyperref}
\usepackage{amssymb, amsmath, amsthm}
\newtheorem{theorem}{Theorem}[section]
\newtheorem{definition}{Definition}[section]
\newtheorem{lemma}{Lemma}[section]
\newtheorem{corollary}{Corollary}[section]
\newtheorem{exa}{Example}
\newtheorem{Remark}{Remark}[section]
\newtheorem{Note}{Note}[section]
\newtheorem{Proposition}{Proposition}[section]
\modulolinenumbers[5]

\journal{Journal of \LaTeX\ Templates}
\newcommand{\least}{\let\cs=\@currsize\renewcommand{\baselinestretch}{.9}\tiny\CS}
\begin{document}
\begin{frontmatter}
	\title{ Space of Invariant bilinear forms under  representation of a group of order 8 }
	
	\author[dam]{Dilchand Mahto \corref{cor1}}\ead{dilchandiitk@gmail.com}
	
	\author[dbm]{Jagmohan Tanti\corref{cor1}}\ead{jagmohan.t@gmail.com}
	
	\address[dam]{Department of Mathematics, Central University of Jharkhand, Ranchi, India}
	\address[dbm]{Department of Mathematics, Babasaheb Bhimrao Ambedkar University, Lucknow, India}
	
	
	
	
	
	
	
	
	\begin{abstract}
		Let $G$ be a  group  of   order $8$ and $\mathbb{F}$  an algebraically closed field with char$(\mathbb{F})\neq 2$. In this paper we compute the  number of $n$ degree representations of  $G$   and subsequent  dimensions of the corresponding spaces of  invatiant bilinear forms  over the field $\mathbb{F}$. We explicitly discuss about the existence of  non-degenerate  invariant  bilinear forms.
	\end{abstract}
	
	\begin{keyword}
		\texttt Invariant bilinear  forms \sep Representation of groups \sep  Vector spaces \sep  Direct sums
		\MSC[2010] 15A63\sep  11E04 \sep 06B15 \sep 15A03
	\end{keyword}
	
\end{frontmatter}


\section{Introduction}
\begin{definition}
	A homomorphism  $\rho$ : $G$ $\rightarrow$ GL($\mathbb{V}(\mathbb{F})$) is called a representation of the group $G$, where $\mathbb{V(\mathbb{F})}$ is a finite dimensional  vector space over $\mathbb{F}$.  $\mathbb{V(\mathbb{F})}$ is also called a representing space of $G$. The dimension of $\mathbb{V}(\mathbb{F})$ over $\mathbb{F}$ is called degree of the representation $\rho$.
\end{definition}
\begin{definition}
	A bilinear form on a finite dimensional vector space $\mathbb{V}(\mathbb{F})$ is said to be invariant under the representation $\rho$ of a finite group $G$ if 
	$$\mathbb{B}(\rho(g)x,\rho(g)y)=\mathbb{B}(x,y), \mbox{ $\forall$ g $\in$ $G$ and x,y $\in\mathbb{V}(\mathbb{F})$}.$$\\
\end{definition}
\noindent Let $\Xi$ denotes the space of bilinear forms on the vector space  $\mathbb{V}(\mathbb{F})$ over $\mathbb{F}$.  
\begin{definition} The space of invariant bilinear forms under the representation $\rho$ is given by
	$$\Xi_{G} = \{\mathbb{B} \in \Xi \,\, |\, \mathbb{B}(\rho(g)x,\rho(g)y)=\mathbb{B}(x,y), \mbox{ $\forall$ $g$ $\in$ $G$ and x,y $\in$ $\mathbb{V}(\mathbb{F})$}\}.$$
\end{definition}

Let  $G$ be a group of order 8, $n \in \mathbb{Z}^{+}$, $\mathbb{F}$ an algebraically closed field with char $\neq$ 2 $\&$ ($\rho$, $\mathbb{V(\mathbb{F})}$)  an $n$ degree representation of $G$ over  $\mathbb{F}$. Then the corresponding  set $\Xi_{G}$ of  invariant bilinear forms on $\mathbb{V(\mathbb{F})}$ under $\rho$, forms a subspace of  $\Xi$. In this paper our investigation pertains to following questions.\\
\\
\textbf{Question.} How many n degree  representations (upto isomosphism) of  $G$  can be  there ? What is the dimension of  $\Xi_{G}$  for every n degree representation ? What are the necessary and sufficient conditions for the existence of a non-degenerate invariant bilinear form.\\

\noindent The  representation ($\rho$, $\mathbb{V(\mathbb{F})}$) is irreducible if there  doesn't exist any  proper invariant  subspace of $\mathbb{V(\mathbb{F})}$ under the  representation $\rho$.   Frobenius (see pp 319, Theorem (5.9) \cite{Artin}) showed that there are only finitely many  irreducible representations  of $G$. Therefore the number of irreducible characters is finite. 
Also by Maschke's theorem ( see pp 316,  corollary (4.9) \cite{Artin}) every $n$ degree representation of $G$ can be written as a direct sum of copies of irreducuble representations $\rho_{i}, i= 1,2,3,..., r$, where $r$ is the number of irreducible representations which is  same as the number of  conjugacy classess \cite{Serr} of  $G$. For  $\rho$ = $\oplus_{i=1}^{r}k_{i}\rho_{i}$ an n degree representation of $G$, the coefficient of $\rho_{i}$ is $k_{i}$, $1 \leq i \leq r$, so that $\sum_{i=1}^{r} d_{i}k_{i}=n, $ and $\sum_{i=1}^{r} d_{i}^{2}= |G|, $ where $d_{i}$ is the degree of $\rho_{i}$  and $d_{i}||G|$ with  $d_{j}\geq d_{i}$ when $j>i$. It is already well understood in the literature that  the   invariant space $\Xi_{G}$ under $\rho$ can be expressed by the set  $\Xi_{G}'$ = $\{X \in \mathbb{M}_{n}(\mathbb{F})\,|\, C_{\rho(g)}^{t}X C_{\rho(g)} = X, \forall g \in G  \}$ with respect to an ordered basis $\underline{e}$ of $\mathbb{V}(\mathbb{F})$, where  $\mathbb{M}_{n}(\mathbb{F})$ is the set of square matrices of order $n$ as the entries from $\mathbb{F}$ and  $C_{\rho(g)}$ is the matrix representation of the linear transformation $\rho(g)$ .\\\\
In the discrete perspective this question has been studied in the literature. Gongopadhyay and Kulkarni \cite{Gong2} investigated the  existence of T-invariant non-degenerate symmetric (resp. skew-symmetric) bilinear forms. Kulkarni and Tanti \cite{Kulk} investigated the  dimension of space of  T-invariant bilinear forms.  Gongopadhyay,  Mazumder and Sardar \cite{Gong} invesigated for  an invertible linear map $T : V \rightarrow V$, when does the vector space
$V$ over $F$ admit a T-invariant non-degenerate c-hermitian form.  Chen \cite{Chen} discussed the all matrix representation of the real numbers.  The application of representation of  group, the character for a reducible representation is important for the  physical problem\cite{DDJ}. For a cubic crystal has many symmetry operations and therefore many
classes and many irreducible representations. The connection between group theory
and quantum mechanics \cite{MGA}, the group of symmetry operators 
which leave the Hamiltonian invariant. These operators 
are symmetry operations of the system and the symmetry operators  commute with the
Hamiltonian. The symmetry operators  are said to form the
group of Schrödinger’s equation.\\

In this paper we investigate about the  counting of n degree representations of a group of order 8, dimensions of their corresponding  spaces of invariant bilinear forms and establish a characterization criteria for existence of a non-degenerate invaiant  bilinear form. Our  investigations  are stated in the following three main theorems. \\
\begin{theorem}
	The number of  n degree representations (upto isomorphism) of a group $G$ of order 8  is $\binom{n+7}{7}$ when $G$ is abelian and  $\sum_{s=0}^{[\frac{n}{2}]}\binom{n-2s+3}{3}$ otherwise.
	\label{theorem1.1}
\end{theorem}

\begin{theorem}
	The space $\Xi_{G}$ of invariant bilinear forms  of a group $G$  of order 8  under an n degree representation $(\rho ,\mathbb{V}(\mathbb{F}))$ is isomorphic to the direct sum of  the subspaces $\mathbb{W}_{(i,j) \in A_G}$ of $\mathbb{M}_{n}(\mathbb{F})$, i.e., $\Xi_{G}'$ = $\bigoplus_{(i,j) \in A_G} \mathbb{W}_{(i,j) \in A_G}$, where $A_G=\{(i,j)\,|\, \, \rho_i$ and $\rho_j$  dual to each other \}  and  $\mathbb{W}_{(i,j) \in A_G}$ = \{ $ X \in \mathbb{M}_{n}(\mathbb{F})\, |\, (i,j)^{th}$  block   $ X_{d_{i}k_{i} \times d_jk_j}^{ij} $  of order $d_{i}k_{i} \times d_jk_j $ satisfying $X_{d_{i}k_{i} \times d_jk_j}^{ij}=  C_{k_{i}\rho_{i}(g)}^{t}X_{d_{i}k_{i} \times d_jk_j}^{ij}C_{k_{j}\rho_{j}(g)}$, $\forall g \in G$  and rest block is zero \}. Also the  dimension of $\mathbb{W}_{(i,j) \in A_G}= k_ik_j$.
	\label{theorem1.2}
\end{theorem}
\begin{theorem}
	(Characterization theorem for an n degree representation of a group of order 8 having a non-degenerate invariant bilinear form).\\
	\textbf{A}.  For $G=D_4$ and $\rho$= $\oplus_{i=1}^{5} k_{i}\rho_{i}$ an $n$ degree representation, $\rho$  always has a non-degenerate invariant bilinear form.\\
	\textbf{B}. 	For $G=Q_8$ and $\rho$= $\oplus_{i=1}^{5} k_{i}\rho_{i}$  an n degree representation, $\rho$ always has a non-degenerate invariant bilinear form.\\
	\textbf{C}. 	For $G=\mathbb{Z}_8$ and $\rho$= $\oplus_{i=1}^{8} k_{i}\rho_{i}$ an n degree representation, $\rho$ has a non-degenerate invariant bilinear form  iff $k_3=k_4, \, \, k_5=k_6 \, \& \, k_7=k_8$ \\
	\textbf{D}. 	For $G=\mathbb{Z}_4 \times \mathbb{Z}_2 $ and $\rho$= $\oplus_{i=1}^{8} k_{i}\rho_{i}$  an n degree representation, $\rho$ has a non-degenerate invariant bilinear  form  iff $k_5=k_6 \, \& \, k_7=k_8.$\\
	\textbf{E}.  	For $G=\mathbb{Z}_2 \times \mathbb{Z}_2 \times \mathbb{Z}_2$ and $\rho$= $\oplus_{i=1}^{8} k_{i}\rho_{i}$ an n degree representation, $\rho$ always  has  a non-degenerate invariant bilinear  form.
	\label{theorem1.3}
\end{theorem}
\begin{Remark}
	Thus we get the necessary and sufficient condition for existence of a non-degenerate invarint bilinear form under an n degree representation from above characterization theorem. It is to remark that these results also hold equally good for a field (not necessarily algebraically closed) of characteristic $\equiv 1\pmod{8}$.
\end{Remark}
This is to note that concerning the existence of a non-degenerate invariant bilinear form on a given representation, it amounts to checking wether this representation is isomorphic to its dual. This can be done by looking at the character tables of the irreducible representations, inverting every character, and checking which dual (irreducible) representation is obtained. Then, in an isomorphism class of representations, specified by multiplicities of irreducible ones, one just needs to check if multiplicities are equal on every irreducible representation and its dual. One can see that at one hand the Theorem \ref{theorem1.3} takes care of all these conversations and on the other hand it has been proved in an elementary way as purely an application of Matrix theory.

\section{Preliminaries }
\noindent In this section  ($\rho$, $\mathbb{V(\mathbb{F})}$)   stands for  an n degree representation over an algebraically closed  field $\mathbb{F}$ with char$\not=2$ of a group $G$ of order 8. Also $\omega \in \mathbb{F}$ is primitive eighth root of unity.

\subsection{Irreducible representations of a group $G$ of order 8} \noindent Here we present the table of irreducible representations for each group of order $8$. We denote the table of $G$ by $T_{G}$.
\subsubsection{$D_4$ = $\langle a, b\,|\,a^{4}=b^{2}=1 , ba= a^{3}b \rangle$}

\begin{center}
	
	$T_{D_{4}} = $ 
	\begin{tabular}{|c|c|c|c|c|c|}
		\hline
		& $\rho_1$&$\rho_2$& $\rho_3$&$\rho_4$&$\rho_5$\\
		\hline
		a& $\omega^{8}$&$\omega^{8}$& $\omega^{4}$&$\omega^{4}$&$\begin{bmatrix}
			0 & \omega^{4} \\
			\omega^{8} & 0
		\end{bmatrix}$\\
		\hline
		b& $\omega^{8}$&$\omega^{4}$& $\omega^{4}$&$\omega^{8}$&$\begin{bmatrix}
			\omega^{8}&0 \\
			0&\omega^{4} 
		\end{bmatrix}$\\ 
		\hline
		
	\end{tabular}.
	
\end{center}

\subsubsection
{$Q_8$ =$\langle a, b\,|\,a^{4}=1 , a^{2}=b^{2}, ba= a^{3}b \rangle$}
\begin{center}
	$T_{Q_{8}} = $ 
	\begin{tabular}{|c|c|c|c|c|c|}
		\hline
		& $\rho_1$&$\rho_2$& $\rho_3$&$\rho_4$&$\rho_5$\\
		\hline
		a& $\omega^{8}$&$\omega^{8}$& $\omega^{4}$&$\omega^{4}$&$\begin{bmatrix}
			\omega^{2}&0 \\
			0 & \omega^{6} 
		\end{bmatrix}$\\
		\hline
		b& $\omega^{8}$&$\omega^{4}$& $\omega^{4}$&$\omega^{8}$&$\begin{bmatrix}
			0&	\omega^{8} \\
			\omega^{4}&0 
		\end{bmatrix}$\\
		\hline
	\end{tabular}.
\end{center}
\subsubsection{$\mathbb{Z}_{8}   $=$\langle a, \,|\,a^{8}=1  \rangle$}
\begin{center}
	$T_{\mathbb{Z}_{8}} = $ 
	\begin{tabular}{|c|c|c|c|c|c|c|c|c|}
		\hline
		& $\rho_1$&$\rho_2$& $\rho_3$&$\rho_4$&$\rho_5$&$\rho_6$&$\rho_7$&$\rho_8$\\
		\hline
		a& $\omega^{8}$&$\omega^{4}$& $\omega$&$\omega^{7}$&$\omega^{2}$&$\omega^{6}$&$\omega^{3}$&$\omega^{5}$\\
		\hline
	\end{tabular}.
\end{center}
\subsubsection{$\mathbb{Z}_{4} \times \mathbb{Z}_{2}  $ =$\langle a, b\,|\,a^{4}= b^{2}=1 ,ab= ba \rangle$}
\begin{center}
	$T_{\mathbb{Z}_{4} \times \mathbb{Z}_{2}} = $ 
	\begin{tabular}{|c|c|c|c|c|c|c|c|c|}
		\hline
		& $\rho_1$&$\rho_2$& $\rho_3$&$\rho_4$&$\rho_5$&$\rho_6$&$\rho_7$&$\rho_8$\\
		\hline
		a& $\omega^{8}$&$\omega^{8}$& $\omega^{4}$&$\omega^{4}$&$\omega^{2}$&$\omega^{6}$&$\omega^{2}$&$\omega^{6}$\\
		\hline
		b& $\omega^{8}$&$\omega^{4}$& $\omega^{8}$&$\omega^{4}$&$\omega^{8}$&$\omega^{8}$&$\omega^{4}$&$\omega^{4}$\\
		\hline
	\end{tabular}.
\end{center}
\subsubsection{$\mathbb{Z}_{2} \times \mathbb{Z}_{2} \times \mathbb{Z}_{2} $ =$\langle a, b,c\,|\,a^{2} =b^{2}=c^{2} =1 ,  ab=ba,ac=ca,bc=cb \rangle$}
\begin{center}
	$T_{\mathbb{Z}_{2} \times \mathbb{Z}_{2} \times \mathbb{Z}_{2}} = $ 	
	\begin{tabular}{|c|c|c|c|c|c|c|c|c|}
		\hline
		& $\rho_1$&$\rho_2$& $\rho_3$&$\rho_4$&$\rho_5$&$\rho_6$&$\rho_7$&$\rho_8$\\
		\hline
		a& $\omega^{8}$&$\omega^{4}$& $\omega^{8}$&$\omega^{8}$&$\omega^{8}$&$\omega^{4}$&$\omega^{4}$&$\omega^{4}$\\
		\hline
		
		b& $\omega^{8}$&$\omega^{4}$& $\omega^{4}$&$\omega^{4}$&$\omega^{8}$&$\omega^{4}$&$\omega^{8}$&$\omega^{8}$\\
		\hline
		
		c& $\omega^{8}$&$\omega^{4}$& $\omega^{8}$&$\omega^{4}$&$\omega^{4}$&$\omega^{8}$&$\omega^{8}$&$\omega^{4}$\\
		\hline
	\end{tabular}.
\end{center}

\begin{definition}
	The  character  of $\rho$ is a  function    $\chi$ : $G$ $\rightarrow$ $\mathbb{F}$,  $\chi(g)=$ tr($\rho(g)$)  and is also called character of the group $G$.
\end{definition}
\begin{theorem}
	(Maschke's Theorem): If char($\mathbb{F}$) does not divide $|G |$, then every representation of  $G$ is a direct sum of irreducible representations.
\end{theorem} 

\begin{proof} 
	See pp 316,   corollary (4.9)  \cite{Artin}.
\end{proof}

\noindent Now as
\begin{equation}
	\rho= k_{1} \rho_{1}\oplus k_{2} \rho_{2}\oplus  ............. \oplus k_{r} \rho_{r},
	\label{directsum}
\end{equation}
where for every $1 \leq i \leq r$, $k_{i} \rho_{i}$ stands for the direct sum of  $k_{i}$ copies of the irreducible representation $\rho_{i}$.

\noindent Let $\chi$ be the corresponding character of the representation $\rho$, then  
$$ \chi= k_{1} \chi_{1}+ k_{2} \chi_{2}+  ............. + k_{r} \chi_{r}.$$
Where $\chi_{i}$ is the character of $\rho_{i}$, $\forall$ $1 \leq i \leq r$.
Dimension of the character $\chi$ is being calculated at identity element of a group. i.e,
\[ dim(\rho) = \chi(1)=tr(\rho(1))\]
\begin{equation}
	\implies d_{1} k_{1} + d_{2}k_{2} +............. +d_{r}k_{r}= n. 
	\label{splitn}
\end{equation}
\begin{Note}
	equation (\ref{splitn}) holds in more general case which  helps us in finding all  possible distinct r-tuples ($k_{1}, k_{2}, ......, k_{r}$), which correspond to the distinct n degree  representations (up to isomorphism) of a finite group.
\end{Note}

\begin{theorem}
	Two representations ($\rho, \mathbb{V(\mathbb{F})}$) and ($\rho', \mathbb{V(\mathbb{F})}$)  are  isomorphic iff their character tables are same i.e, $\chi(g)=\chi'(g)$ for all g $\in G$.
	\label{theorem2.2}
\end{theorem}

\begin{proof}
	See  pp 319,  corollary ( 5.13) \cite{Artin}.
\end{proof}
\section{Existence of non-degenerate invariant bilinear  form.}
\indent The space of invariant  bilinear forms under a finite  group may have   non-degenerate and  degenerate forms. Sometimes  all the elements of the space are degenerate with respect to  a particular   representation, such a space is called a degenerate invariant space in this paper. How many such repsentations exist out of  total representations, is a matter of investigation. Some of the spaces contains both non-degenerate and degenerate invariant  bilinear forms under   a particular representation. In this section we compute the number of such representations of the group $G$ of order 8.\\
\begin{Remark}
	The space $\Xi_{G}'$  of invariant bilinear forms under an n degree reprsentation $\rho$ contains only those $X \in \mathbb{M}_{n}(\mathbb{F})$ whose  $(i,j)^{th}$ block is a O sub-matrix of order ${d_{i}k_{i}\times d_{j}k_{j}}$ when $(i,j) \notin A_G=\{(i,j)\,|\,\, \rho_i$  and $\rho_j$  is dual to each other \}  whereas  the  $(i,j)^{th}$ block of order $d_ik_i \times d_jk_j$, is given by
	
	For $G \neq \mathbb{Z}_8$ and $(i,j) = (1,1),(2,2),(3,3),(4,4)$, it is 
	$$X^{ij}_{d_ik_i \times d_jk_j}= 
	\begin{bmatrix}
		x^{ij}_{11}&x^{ij}_{12}&...&...&...&x^{ij}_{1k_{j}}\\
		x^{ij}_{21}&x^{ij}_{22}&...&...&...&x^{ij}_{2k_{j}}\\
		...&...&...&...&...&...\\
		...&...&..&...&...&...\\
		x^{ij}_{k_{i}1}&x^{ij}_{k_{i}2}&...&...&...&x^{ij}_{k_{i}k_{j}}\\
	\end{bmatrix},$$
	and for $G=\mathbb{Z}_8$ and  $(i,j)= (1,1),(2,2),(3,4),(4,3),(5,6),(6,5),(7,8),(8,7)$, it is 
	
$$X^{ij}_{d_ik_i \times d_jk_j}= 
	\begin{bmatrix}
		x^{ij}_{11}&x^{ij}_{12}&...&...&...&x^{ij}_{1k_{j}}\\
		x^{ij}_{21}&x^{ij}_{22}&...&...&...&x^{ij}_{2k_{j}}\\
		...&...&...&...&...&...\\
		...&...&..&...&...&...\\
		x^{ij}_{k_{i}1}&x^{ij}_{k_{i}2}&...&...&...&x^{ij}_{k_{i}k_{j}}\\
	\end{bmatrix},$$
	for $G=\mathbb{Z}_4 \times \mathbb{Z}_2$ and  $(i,j)= (5,6),(6,5),(7,8),(8,7)$, it is

$$X^{ij}_{d_ik_i \times d_jk_j}= 
	\begin{bmatrix}
		x^{ij}_{11}&x^{ij}_{12}&...&...&...&x^{ij}_{1k_{j}}\\
		x^{ij}_{21}&x^{ij}_{22}&...&...&...&x^{ij}_{2k_{j}}\\
		...&...&...&...&...&...\\
		...&...&..&...&...&...\\
		x^{ij}_{k_{i}1}&x^{ij}_{k_{i}2}&...&...&...&x^{ij}_{k_{i}k_{j}}\\
	\end{bmatrix},$$
	
	for  $(i,j)=(5,5)$,  with  $ G= D_4$ it is 
	$$X^{ii}_{d_ik_i}= 
	\begin{bmatrix}
		x^{ii}_{11}I_{2}&x^{ii}_{13}I_{2}&...&...&...&x^{ii}_{1(2k_{i}-1)}I_{2}\\
		x^{ii}_{31}I_{2}&x^{ii}_{33}I_{2}&...&...&...&x^{ii}_{3(2k_{i}-1)}I_{2}\\
		...&...&...&...&...&...\\
		...&...&...&...&...&...\\
		x^{ii}_{(2k_{i}-1)1}I_{2}&x^{ii}_{(2k_{i}-1)2}I_{2}&...&...&...&x^{ii}_{(2k_{i}-1)(2k_{i}-1)}I_{2}\\
	\end{bmatrix},
	$$
	$\&$ for  $(i,j)=(5,5)$, with $G=Q_8$ it is
	$$X^{ii}_{d_ik_i}= 
	\begin{bmatrix}
		x^{ii}_{12}I_{2}^{-}&x^{ii}_{14}I_{2}^{-}&...&...&...&x^{ii}_{1(2k_{i})}I_{2}^{-}\\
		x^{ii}_{32}I_{2}^{-}&x^{ii}_{34}I_{2}^{-}&...&...&...&x^{ii}_{3(2k_{i})}I_{2}^{-}\\
		...&...&...&...&...&...\\
		...&...&...&...&...&...\\
		x^{ii}_{(2k_{i}-1)2}I_{2}^{-}&x^{ii}_{(2k_{i}-1)4}I_{2}^{-}&...&...&...&x^{ii}_{(2k_{i}-1)(2k_{i})}I_{2}^{-}\\
	\end{bmatrix},
	$$
	where $I_{2}^{-}=\begin{bmatrix}
		0&1\\-1&0
	\end{bmatrix}$.
	\label{remark}
\end{Remark} 
\begin{Note}
	Note that $X\in \Xi_G'$ is invariant under $\rho$ if and only if for every $(i,j) \in A_G$, we have
	$X^{ij}_{d_{i}k_{i} \times d_{j}k_{j} }=  C_{k_{i}\rho_{i}(g)}^{t}X^{ij}_{d_{i}k_{i} \times d_{j}k_{j} }C_{k_{j}\rho_{j}(g)}$,  $\forall g \in G$ and $X=[ X^{ij}_{d_{i}k_{i} \times d_{j}k_{j}} ]_{(i,j) \in A_G}.$
\end{Note}
\begin{lemma}
	If $X \in \Xi_{G}'$, then $X^{ij}_{d_{i}k_{i} \times d_{j}k_{j}}$ is non-singular,  with $(i,j) \in A_G$ iff  $X$ is non-singular.
	\label{lemma3.0}
\end{lemma}  

\begin{proof}  With reference to the above remark and note, for every $X \in \Xi_{G}'$, we have  $X^{ij}_{d_{i}k_{i} \times d_{j}k_{j} }=  C_{k_{i}\rho_{i}(g)}^{t}X^{ij}_{d_{i}k_{i} \times d_{j}k_{j} }C_{k_{j}\rho_{j}(g)}$,  $\forall g \in G$. Suppose $X^{ij}_{d_{i}k_{i} \times d_{j}k_{j}}$ is non-singular, then  $X^{ij}_{d_{i}k_{i} \times d_{j}k_{j}}$ is square sub-matix   for $(i,j) \in A_G$ and   rows (columns) of X is linearly linearly independent. Thus the result follows.  Converse part is easy to see.
\end{proof}
To prove  following lemmas, from remark  \ref{remark}  we will choose  only those $X \in \mathbb{M}_{n}(\mathbb{F})$ whose $(i,j)^{th} $ block is zero for $(i,j) \notin A_G$ and for $(i,j) \in A_G$,   the  $(i,j)^{th}$ block  $X^{ij}_{d_{i}k_{i} \times d_{j}k_{j}}$ is    non-singular. 

\begin{lemma} For $n \in \mathbb{Z}^{+}$, the number of n degree representations of group $D_4$ or $Q_8$, whose corresponding spaces of invariant bilinear forms contain non-degenerate bilinear froms is  $\sum_{s=0}^{[\frac{n}{2}]}\binom{n-2s+3}{3}$.
	\label{lemma3.1}
\end{lemma}
\begin{proof}
	From (\ref{splitn}) we have $ k_{1} + k_{2} +k_{3} +k_{4} +2k_{5}= n$ and we have to choose    $X \in \mathbb{M}_{n}(\mathbb{F})$ such that  $X= Diag\big[X^{11}_{k_{1}}, X^{22}_{k_{2}}, X^{33}_{k_{3}},X^{44}_{k_{4}}, X^{55}_{2k_{5}} \big]$. For $1 \leq i \leq 5$, the chosen sub-matrices $X^{ii}_{d_ik_i}$  is  non-singular. Thus  $X= Diag\bigg[X^{11}_{k_{1}},$ $ X^{22}_{k_{2}}, X^{33}_{k_{3}},X^{44}_{k_{4}},X^{55}_{2k_{5}} \bigg]$ is non-singular and $ X = C_{\rho(g)}^{t}XC_{\rho(g)}$, $\forall g \in D_{4}$ or $Q_8$ with $$ k_{1} +k_2 +k_{3}  +k_{4}+ 2k_{5}= n.
	$$

	\noindent  We have a collection of distinct   5-tuples ($k_{1}, k_{2}, k_{3},k_{4}, k_{5}$) of size  $\sum_{s=0}^{[\frac{n}{2}]}\binom{n-2s+3}{3}$,
	which is same as the number of representations of degree $n$, whose corresponding spaces contain non-degenerate bilinear forms.
\end{proof}
\begin{lemma}
	For every $n \in \mathbb{N}$, the number of n degree representations of group $ \mathbb{Z}_{8}$, whose corresponding spaces of invariant bilinear forms contain  non-degenerate  bilinear froms is $\sum_{s=0}^{[\frac{n}{2}]}\binom{s+2}{2}\binom{n-2s+1}{1}$.
	\label{lemma3.3}
\end{lemma}

\begin{proof} Let  $G = \mathbb{Z}_{8}$. Then $ k_{1} + k_{2} +k_{3} +k_{4} +k_{5}+k_{6} +k_{7} +k_{8}= n $ and  we have to choose    $X \in \mathbb{M}_{n}(\mathbb{F})$ such that.
	$$X= Diag\bigg[X^{11}_{k_1 \times k_{1}}  ,X^{22}_{k_{2}\times k_{2}}, \begin{bmatrix}
		O&X^{34}_{k_{3} \times k_4} \\
		X^{43}_{k_{4}\times k_3}&O
	\end{bmatrix}, \begin{bmatrix}
		O&X^{56}_{k_{5}\times k_6} \\
		X^{65}_{k_{6}\times k_5}&O
	\end{bmatrix} , \begin{bmatrix}
		O&X^{78}_{k_{7}\times k_8} \\
		X^{87}_{k_{8}\times k_7}&O
	\end{bmatrix} \bigg]$$
	
	with   $k_3 =k_4 \, \, , k_5= k_6$ $\&$ $ k_7=k_8$. Since the chosen sub-matrices $X^{ij}_{d_ik_i \times d_jk_j}$  is  non-singular for $(i,j) \in \{(1,1),(2,2),(3,4),(4,3),(5,6),(6,5),(7,8),(8,7) \}$.  The rows or columns of  $X$ are linearly independent thus $X$ is non-singular and $ X = C_{\rho(g)}^{t}XC_{\rho(g)}$, $\forall g \in \mathbb{Z}_8$ with  
	$$ k_{1} +k_2+2k_3+2k_5+2k_7= n.$$ 
	As the number of  such  5-tuples ($k_{1}, k_{2}, k_{3},k_{5}, k_{7}$) of size  $\sum_{s=0}^{[\frac{n}{2}]}\binom{s+2}{2}\binom{n-2s+1}{1}$, the number of representations of degree $n$, whose corresponding spaces contain non-degenerate  bilinear forms is $\sum_{s=0}^{[\frac{n}{2}]}\binom{s+2}{2}\binom{n-2s+1}{1}$.
\end{proof}
\begin{lemma}
	For every $n \in \mathbb{N}$, the  number of n degree representations of  the group $ \mathbb{Z}_{4} \times \mathbb{Z}_{2} $, whose corresponding spaces of invariant bilinear forms contain non-degenerate bilinear froms is $\sum_{s=0}^{[\frac{n}{2}]}\binom{s+1}{1}\binom{n-2s+3}{3}$.
	\label{lemma3.4}
\end{lemma}

\begin{proof}
	Let  $G = \mathbb{Z}_{4} \times \mathbb{Z}_{2} $. Then $ k_{1} + k_{2} +k_{3} +k_{4} +k_{5}+k_{6} +k_{7} +k_{8}= n $  and  we have to choose    $X \in \mathbb{M}_{n}(\mathbb{F})$ such that  	$$X= Diag\bigg[X^{11}_{k_1 \times k_{1}}  ,X^{22}_{k_{2}\times k_{2}}, ,X^{33}_{k_{3}\times k_{3}} ,X^{44}_{k_{4}\times k_{4}}, \begin{bmatrix}
		O&X^{56}_{k_{5}\times k_6} \\
		X^{65}_{k_{6}\times k_5}&O
	\end{bmatrix} , \begin{bmatrix}
		O&X^{78}_{k_{7}\times k_8} \\
		X^{87}_{k_{8}\times k_7}&O
	\end{bmatrix} \bigg]$$
	
	with   $ k_5= k_6$ $\&$ $ k_7=k_8$. Since the chosen sub-matrices $X^{ij}_{d_ik_i \times d_jk_j}$  is  non-singular for $(i,j) \in \{(1,1),(2,2),(3,3),$ $(4,4),(5,6),(6,5),(7,8),(8,7) \}$.  The rows or columns of  $X$ are linearly independent thus $X$ is non-singular and $ X = C_{\rho(g)}^{t}XC_{\rho(g)}$, $\forall g \in \mathbb{Z}_4 \times \mathbb{Z}_2$ with  
	$$ k_{1} +k_2+k_3+k_4+2k_5+2k_7= n.$$
	As the number of  such  6-tuples ($k_{1}, k_{2}, k_{3},k_4, k_{5}, k_{7}$) of size  $\sum_{s=0}^{[\frac{n}{2}]}\binom{s+1}{1}\binom{n-2s+3}{3}$,  the number of representations of degree $n$, whose corresponding spaces contain non-degenerate  bilinear forms is $\sum_{s=0}^{[\frac{n}{2}]}\binom{s+1}{1}\binom{n-2s+3}{3}$.
\end{proof} 
\begin{lemma}
	For every $n \in \mathbb{N}$, the number of n degree representations of a group $ \mathbb{Z}_{2} \times \mathbb{Z}_{2} \times \mathbb{Z}_{2} $, whose corresponding spaces of invariant bilinear forms contain  non-degenerate  bilinear forms is $\binom{n+7}{7}$.
	\label{lemma3.5}
\end{lemma} 

\begin{proof}
	From the equation  (\ref{splitn}), we have  $ k_{1} + k_{2} +k_{3} +k_{4} +k_{5}+k_{6} +k_{7} +k_{8}= n $    and  we have to  choose    $X \in \mathbb{M}_{n}(\mathbb{F})$ such that  $X= Diag\bigg[X^{11}_{k_{1}}, X^{22}_{k_{2}}, X^{33}_{k_{3}},X^{44}_{k_{4}}, X^{55}_{k_{5}}, X^{66}_{k_{6}},$ $X^{77}_{k_{7}}, X^{88}_{k_{8}} \bigg]$. For  $1\leq i \leq 8$, the  chosen sub-matrices $X^{ii}_{d_ik_i \times d_ik_i}$  is  non-singular.   The rows or columns of  $X$ are linearly independent thus $X$ is non-singular and $ X = C_{\rho(g)}^{t}XC_{\rho(g)}$, $\forall g \in \mathbb{Z}_{2} \times \mathbb{Z}_{2} \times \mathbb{Z}_{2}$ with  $$ k_{1}  +k_{2}  +k_{3}+k_4+k_{5}  +k_{6}  +k_{7}+k_8= n.
	$$ 
	As the number of  such  8-tuples  $(k_1 , k_2, k_3, k_4,k_5,k_6,k_7,k_8)$  is $\binom{n+7}{7}$,  the number of representations of degree $n$, whose corresponding spaces contain non-degenerate  bilinear forms is $\binom{n+7}{7}$.	
\end{proof} 
\begin{Remark} Since $\mathbb{F}$ is algebraically closed, it has infinitely many non zero elements, hence  if there is one non-degenerate  invariant bilinear form in the space $\Xi_{G}$, it has   infinitely many.\\
\end{Remark}
Thus from Lemmas  $\ref{lemma3.1}$ to $\ref{lemma3.5}$, we find that the number of $n$ degree representations of a group $G$ of order 8, whose corresponding  spaces of invariant bilinear forms   contain non-degenerate forms  are $\sum_{s=0}^{[\frac{n}{2}]}\binom{n-2s+3}{3}$, $\sum_{s=0}^{[\frac{n}{2}]}\binom{n-2s+3}{3}$, $\sum_{s=0}^{[\frac{n}{2}]}\binom{s+1}{1}\binom{n-2s+3}{3}$, $\sum_{s=0}^{[\frac{n}{2}]}\binom{s+2}{2}\binom{n-2s+1}{1}$ and $\binom{n+7}{7}$ corresponding to  the  group  $D_{4}$,  $Q_{8}$,  $\mathbb{Z}_{2} \times \mathbb{Z}_{4}$, $\mathbb{Z}_8$ and $\mathbb{Z}_{2} \times \mathbb{Z}_{2} \times \mathbb{Z}_{2}$  respectively. 
\\
\subsection{Characterization of invariant bilinear forms under an n degree representation of a group of order 8}
\begin{lemma}
	For $G=D_4$ and $\rho$= $\oplus_{i=1}^{5} k_{i}\rho_{i}$ an $n$ degree representation of $G$,  $\rho$  always has a non-degenerate bilinear form.
	\label{lemma3.6}
\end{lemma}
\begin{proof} Follows from the proof of Lemma $\ref{lemma3.1}$.
\end{proof}
\begin{lemma}
	For $G=Q_8$ and $\rho$= $\oplus_{i=1}^{5} k_{i}\rho_{i}$  an n degree representation of $G$,  $\rho$ always has a non-degenerate bilinear form.
	\label{lemma3.7}
\end{lemma}
\begin{proof} Follows from the proof of Lemma $\ref{lemma3.1}$.
\end{proof}
\begin{lemma}
	For $G=\mathbb{Z}_8$ and $\rho$= $\oplus_{i=1}^{8} k_{i}\rho_{i}$ an n degree representation of $G$, $\rho$ has a non-degenerate bilinear form iff $k_3  =k_4,  k_5=k_6 \, \& \, k_7=k_8.$
	\label{lemma3.8}
\end{lemma}
\begin{proof} Follows from the proof of  Lemma $\ref{lemma3.3}$.
\end{proof}
\begin{lemma}
	For $G=\mathbb{Z}_4 \times \mathbb{Z}_2 $ and $\rho$= $\oplus_{i=1}^{8} k_{i}\rho_{i}$  an n degree representation of $G$,  $\rho$ has a non-degenerate bilinear form iff $k_5=k_6 \, \& \, k_7=k_8.$
	\label{lemma3.9}
\end{lemma}
\begin{proof}Follows from the proof of Lemma $\ref{lemma3.4}$.
\end{proof}
\begin{lemma}
	For $G=\mathbb{Z}_2 \times \mathbb{Z}_2 \times \mathbb{Z}_2$ and $\rho$= $\oplus_{i=1}^{8} k_{i}\rho_{i}$ an n degree representation of $G$,  $\rho$ always has a non-degenerate bilinear form.
	\label{lemma3.10}
\end{lemma}
\begin{proof} Follows from the proof of Lemma $\ref{lemma3.5}$.
\end{proof}
\begin{definition}
	The space $\Xi_{G}$ of invariant bilinear forms  is   called  degenerate  if it's all elements are degenerate.
\end{definition}
We  will  discuss about the degenerate  invariant  space in the later section.\\

\section{Dimensions of  spaces of  invariant bilinear forms under representations of  groups of order 8.}
\noindent  The  space of invariant   bilinear forms under an n degree representation is finite  dimensional and so are the  symmetric subspace and the skew-symmetric subspace. In this section  we calculate  the dimension of the space of invariant bilinear forms under a representation of a group of order 8.\\
\begin{theorem}
	If $\Xi_{G}$ is the  space of invariant bilinear forms  under an n degree representation  $\rho =\oplus_{i=1}^{r}k_{i}\rho_{i}$ of a group $G$ of order 8, then  dim$(\Xi_{G})=\sum_{(i,j) \in A_G}k_{i}k_j$. 	 
	\label{theorem4.1}
\end{theorem}
\begin{proof}
	Since for every   $X \in \Xi_{G}'$,  the non zero sub-matrices    $X^{ij}_{d_{i}k_{i} \times d_{j}k_{j}}$,  for each $(i,j) \in A_G$ we have,   $X^{ij}_{d_{i}k_{i} \times d_{j}k_{j}}$ = $C_{k_{i}\rho_{i}(g)}^{t}$$X^{ij}_{d_{i}k_{i} \times d_{j}k_{j}}$$C_{k_{j}\rho_{j}(g)}$. Now with reference of remark $\ref{remark}$ to span  $(i,j)^{th}$ sub-matrix of  $ X \in \Xi_{G}'$ it  needs maximum $k_{i}k_j$ linearly indepedent vectors from $\mathbb{M}_{n}(\mathbb{F})$. This completes the proof.
\end{proof}
\begin{corollary}
	The  space of invariant symmetric bilinear forms  under an n degree representation  $\rho =\oplus_{i=1}^{r}k_{i}\rho_{i}$ of a group $G \neq Q_8$ of order 8 has dimension $=\sum_{(i,i) \in A_G}\frac{k_{i}(k_{i}+1)}{2} + \sum_{\substack{(i,j) \in {A_G} \\ i \neq j}}\frac{k_ik_j}{2}$.
\end{corollary}
\begin{proof}
	Follows  from the proof of theorem $\ref{theorem4.1}$ and remark $\ref{remark}$.
\end{proof}
\begin{corollary}
	The space of invariant skew-symmetric bilinear forms  under an n degree representation  $\rho =\oplus_{i=1}^{r}k_{i}\rho_{i}$ of a group $G \neq Q_8$ of order 8 has  dimension $=\sum_{(i,i) \in A_G}\frac{k_{i}(k_{i}-1)}{2} + \sum_{\substack{(i,j) \in {A_G} \\ i \neq j}}\frac{k_ik_j}{2}$.
\end{corollary}

\begin{proof}
	Follows  from the proof of theorem $\ref{theorem4.1}$  and remark $\ref{remark}$.
\end{proof}
\begin{corollary}
	The space of invariant symmetric  (skew-symmetric) bilinear forms  under an n degree representation  $\rho =\oplus_{i=1}^{r}k_{i}\rho_{i}$ of a group $ Q_8$  has  dimension $=\sum_{i=1}^{4}\frac{k_{i}(k_{i} \pm 1)}{2}+ \frac{k_{5}(k_{5} \mp 1)}{2}$.
\end{corollary}

\begin{proof}
	Follows  from the proof of theorem $\ref{theorem4.1}$  and remark $\ref{remark}$.
\end{proof}

\section{Main results \label{proof1.1}}
Here we present the proofs of main theorems stated in the Introduction section.
\begin{flushleft}
	\textbf{Proof of theorem $\ref{theorem1.1}$}	
	Since $G$ is a group of order 8 and dimension of the  vector space  $\mathbb{V}(\mathbb{F})$ 	is n, if $G$ is $D_{4}$ or $Q_{8}$ then  r=5, $d_{i}$ =1 for i=1,2,3,4 and $d_{5}$ = 2. Now from equation  (\ref{splitn}) we have 
	$$ k_{1} + k_{2} +k_{3} +k_{4} +2k_{5}= n. $$

\end{flushleft}
From the proof of the lemma $\ref{lemma3.1}$  the number of  distinct   5-tuples ($k_{1}, k_{2}, k_{3},k_{4}, k_{5}$) is  $\sum_{s=0}^{[\frac{n}{2}]}\binom{n-2s+3}{3}$.\\
If G any of $\mathbb{Z}_{2} \times \mathbb{Z}_{2}\times \mathbb{Z}_{2}$,  $\mathbb{Z}_{4} \times \mathbb{Z}_{2}$,  $\mathbb{Z}_{8} $  we have  $r=8$ and  $d_{i} =1$ for $\displaystyle i=1,2,....,8$. Now from  equation (\ref{splitn}), we have 
$$ k_{1} + k_{2} +k_{3} +k_{4} +k_{5}+k_{6} +k_{7} +k_{8}= n, $$
so the  number of such  8-tuples ($k_{1}, k_{2}, k_{3},k_{4},k_{5},k_{6} ,k_{7}, k_{8}$) is $\binom{n+7}{7}$.\\
Thus from (\ref{splitn}) and Theorem $\ref{theorem2.2}$ the  number of  n degree representations (upto isomorphism) of a group $G$ of order 8  is $\sum_{s=0}^{[\frac{n}{2}]}\binom{n-2s+3}{3}$ for non abelian  and  $\binom{n+7}{7}$ for abelian.
$\hspace{2.4in}$  $\square$\\
\subsection{Degenerate invariant  spaces} \noindent From  Theorem $\ref{theorem1.1}$ and Lemmas $\ref{lemma3.1}$ to $\ref{lemma3.5}$, we have the number of $n$ degree representations whose corresponding  invariant  spaces of bilinear forms   contain only degenerate invariant bilinear forms are  $\binom{n+7}{7}-$$\sum_{s=0}^{[\frac{n}{2}]}\binom{s+1}{1}\binom{n-2s+3}{3}$, $\binom{n+7}{7}-$ $\sum_{s=0}^{[\frac{n}{2}]}\binom{s+2}{2}\binom{n-2s+1}{1}$, 0, 0, 0 
of the  groups    $\mathbb{Z}_{2} \times \mathbb{Z}_{4}$, $\mathbb{Z}_{8} $, $\mathbb{Q}_{8}$, $D_{4}$ and   $\mathbb{Z}_{2} \times \mathbb{Z}_{2} \times \mathbb{Z}_{2}$,   respectively. 
The   groups     $\mathbb{Z}_{8} $   $\&$  $\mathbb{Z}_{2} \times \mathbb{Z}_{4}$   have  representations whose corresponding spaces  of invariant forms  are  degenerate and the groups  $\mathbb{Z}_{2} \times \mathbb{Z}_{2} \times \mathbb{Z}_{2}$, $D_{4}$ $\&$    $Q_{8}$ have no  degenerate spaces.\\

\noindent \textbf{Proof of  theorem $\ref{theorem1.2}$}
Let $\mathbb{W}_{(i,j) \in A_G}$ be the subspaces of $\mathbb{M}_{n}(\mathbb{F})$ and $\Xi_{G}'$  the space of invariant bilinear forms of a group $G$. Let $X$ be an element of $\Xi_{G}'$ then 

\[C_{\rho(g)} ^{t}X C_{\rho(g)}= X \,\,and\,\, X=[ X^{ij}_{d_{i}k_{i} \times d_{j}k_{j}} ]_{(i,j) \in A_G} \]
\noindent Existence:\\
Let $X \in \Xi_{G}' $ then for every   $ (i,j) \in A_G$,  there exists at least one  $X_{(i,j)} \in \mathbb{W}_{(i,j) \in A_G}$, such that  $\sum_{(i,j) \in A_G} X_{(i,j)} =X$.\\

\noindent Uniqueness: \\
For every  $ (i,j) \in A_G$, suppose there exists  $Y_{(i,j)} \in \mathbb{W}_{(i,j) \in A_G} $, such that   $\sum_{(i,j) \in A_G}Y_{(i,j)} =X$, then  $\sum_{(i,j) \in A_G} X_{(i,j)} $ = $\sum_{(i,j) \in A_G} Y_{(i,j)}$ i.e.,  $Y_{(i',j')}-X_{(i',j')}$= $\sum_{(i,j) \neq (i',j')} (X_{(i,j)}-Y_{(i,j)})$. Therefore $Y_{(i',j')}-X_{(i',j')} \in $ $\sum_{(i,j) \neq (i',j')}  \mathbb{W}_{(i,j) \in A_G}  $ hence  $Y_{(i',j')}-X_{(i',j')}$ = O or $Y_{(i',j')}=X_{(i',j')}$ for all $(i',j') \in A_G$.\\Thus we have 
\begin{equation}
	\Xi_{G}' =\oplus_{(i,j) \in A_G} \mathbb{W}_{(i,j) \in A_G} \hspace{0.1cm} and \hspace{0.1cm} dim(\Xi_{G}') = \sum_{(i,j) \in A_G} dim(\mathbb{W}_{(i,j) \in A_G}).
	\label{wg}
\end{equation}

\begin{flushleft}
	Now 
	as $\mathbb{W}_{(i,j) \in A_G} = \{ X \in \mathbb{M}_{n}(\mathbb{F})\, |\,  \,  \, (i,j)^th$ block  $X_{d_{i}k_{i} \times d_j k_j}^{ij}$  a  sub - matrix of order $d_{i}k_{i} \times d_jk_j $ satisfying $X_{d_{i}k_{i} \times d_jk_j}^{ij}=  C_{k_{i}\rho_{i}(g)}^{t}X_{d_{i}k_{i} \times d_jk_j}^{ij}C_{k_{j}\rho_{j}(g)}$, $\forall g \in G$ and rest block is zero \},
	from the remark  \ref{remark}, we see that for  $(i,j) \in A_G$, the sub-matrices  $X_{d_{i}k_{i} \times d_jk_j}^{ij}$ in  $\mathbb{W}_{(i,j) \in A_G} $  have $k_i k_j$ free variables $\&$  $\mathbb{W}_{(i,j) \in A_G} \cong $ $\mathbb{M}_{k_i \times k_j }(\mathbb{F})$. Thus $\Xi_{G}'\cong $ $\oplus_{(i,j) \in A_G}\mathbb{M}_{k_i \times k_j}(\mathbb{F})$ and $dim(\mathbb{W}_{(i,j) \in A_G} )= k_{i}k_j$.
\end{flushleft}


Thus  substituting this in equation  (\ref{wg}) we get the dimension of $\Xi_{G}'$. \\

\noindent \textbf{Proof of  theorem $\ref{theorem1.3}$} Follows immediately from  Lemmas $\ref{lemma3.6}$ to $\ref{lemma3.10}$ .$\hspace{1.9 in}$ $\square$\\
\section { Representations over a field of characteristic 2.}
\begin{Remark}
	If characteristic of the field $\mathbb{F}$ is 2 then a  group $G$ of order $8$ has only  trivial irreducible representation.
	Therfore  n copies of irreducible  representation is written  as  
	$$\rho(g)=n\rho_{1}(g),$$
	where $\rho_{1}$ is the  trivial representation of a group $G$ of degree 1. So  the  representation $\rho$ is a trivial  representation of degree n. i.e,  \\
	$$  \rho(g)= I_{n},\mbox{\, for  all $g \in G$.} $$
\end{Remark}
\begin{Proposition}  The space of invariant bilinear forms under an n degree trivial representation of a group $G$ of order $8$ with char($\mathbb{F}$) = 2 is  isomorphic to  $\mathbb{M}_{n}(\mathbb{F})$. 
\end{Proposition} 	

\begin{Proposition}  The space of symmetric invariant bilinear forms is the  direct sum of space of  skew-symmetric invariant forms and space of diagonal invariant forms.
\end{Proposition}

\begin{Note}  If char($\mathbb{F})=2$, the Maschke's theorem does not hold and there are composite representations that are not direct sums.
\end{Note}

\vskip2mm
Thus here we have completely characterised the representations of a group of order $8$ for having a non-degenerate invariant bilinear form over an algebraically closed field. Note that these results hold equally good when considered over a field of characteristic $\equiv1\pmod{8}$.
\vskip2mm
\noindent {\bf Acknowledgement} The first author would like to thank UGC, India for providing the research fellowship and authors are thankful to the Central University of Jharkhand, India for support to carry out this research work. The second author is thankful to the Babasaheb Bhimrao Ambedkar University, Lucknow, India where he revised the paper.


\end{document}